\documentclass{amsart}

\usepackage{amsmath, amssymb, amsthm}
\PassOptionsToPackage{hyphens}{url}
\usepackage{hyperref}
\usepackage{microtype}

\newtheorem{theorem}{Theorem}
\newtheorem{lemma}[theorem]{Lemma}
\newtheorem{corollary}[theorem]{Corollary}
\newtheorem{proposition}[theorem]{Proposition}
\theoremstyle{definition}
\newtheorem{definition}[theorem]{Definition}
\newtheorem{remark}[theorem]{Remark}
\newtheorem{claim}[theorem]{Claim}

\newcommand{\D}{\mathbb{D}}
\newcommand{\Z}{\mathbb{Z}}
\newcommand{\Zi}{\mathbb{Z}[i]}
\newcommand{\C}{\mathbb{C}}
\newcommand{\R}{\mathbb{R}}
\newcommand{\abs}[1]{\left|#1\right|}
\newcommand{\OD}{\mathcal{O}(\D)}
\newcommand{\mR}{\mathcal{R}}
\newcommand{\mRR}{\mathcal{R}_{\mathbb{R}}}

\title{Integer-Coefficient Power Series with Prescribed Zero Sets}

\author{Jon Bannon}
\address{Department of Mathematics and Statistics, Siena University,
  Loudonville, NY 12211}
\email{jbannon@siena.edu}

\author{David Feldman}
\address{Department of Mathematics and Statistics, The University of
  New Hampshire, Durham, NH 03824}
\email{david.feldman@unh.edu}


\thanks{Theorems~\ref{thm:main} and~\ref{prop:Zi} have been formalized
  and machine-checked in the Lean~4 proof assistant (no
  \texttt{sorry}s, no axioms beyond Mathlib's standard
  \texttt{propext}, \texttt{Classical.choice}, and \texttt{Quot.sound}):
  see \url{https://github.com/JonBannon/WeierstrassFormalization},
  with Theorem~\ref{thm:main} proved in
  \url{https://github.com/JonBannon/WeierstrassFormalization/blob/master/WeierstrassFormalization/MainTheorem.lean}.}
\date{}

\begin{document}

\begin{abstract}
We prove that a discrete effective divisor on the open unit disk $\D$
is the zero divisor of a holomorphic function on $\D$ with integer
Taylor coefficients if and only if it is invariant under complex
conjugation.  The construction uses a one-parameter deformation of the
Weierstrass elementary factors in which each modified factor of order
$n$ leaves all Taylor coefficients of degree $\leq n$ unchanged while
shifting the coefficient of degree $n+1$ by a controlled affine amount.
These modified factors act as elementary jet-correction operators: the
triangular structure of the coefficient map permits an inductive
rounding scheme compatible with canonical-product convergence.  As a
consequence, every holomorphic function on $\D$ differs from one with
Gaussian-integer Taylor coefficients by multiplication by a
nowhere-vanishing holomorphic factor.
\end{abstract}

\maketitle

\section{Introduction}

Transcendental number theory rests on the special nature of roots of
polynomials with integer coefficients.  Considering instead roots of
power series with integer coefficients changes the picture in two
immediate ways.  First, a non-polynomial integer-coefficient
power series has radius of convergence at most~$1$, so if it converges
anywhere in $\C \setminus \Z$ it converges inside the unit disk $\D =
\{z \in \C : \abs{z} < 1\}$, and analytic hypotheses must enter the
story.  Second, the set of such series convergent on $\D$ is
uncountable, and any single point $z_0 \in \D$ is a root of infinitely
many of them: the evaluation map $f \mapsto f(z_0)$ is an additive
homomorphism on the uncountable group $\Z[[z]] \cap \OD$, and its
kernel is correspondingly large (Theorem~\ref{thm:main}, applied to the
conjugation-stable divisor $[z_0] + [\bar z_0]$, already produces such
series, in infinitely many ways since the construction is
non-canonical).  Individual roots thus carry no arithmetic content.

The natural analogue asks which zero sets of holomorphic functions
on $\D$ arise as the zero sets of integer power series.  An
\emph{effective divisor} on $\D$
is a locally finite formal sum $D = \sum_{a \in S} m_a [a]$ with
non-negative integer multiplicities over a discrete subset $S \subset
\D$, encoding a zero set with multiplicities.  For $D$ to be the zero
divisor of an integer-coefficient holomorphic function, the discreteness
of $S$ is necessary (and automatic), and $D$ must be \emph{invariant
under complex conjugation} ($m_{\bar a} = m_a$ for all $a$), since
$f(\bar z) = \overline{f(z)}$ for real-coefficient power series forces
nonreal zeros into conjugate pairs.  Conjugation invariance turns out
to be the only obstruction: it is also sufficient.

\begin{theorem}
\label{thm:main}
An effective divisor $D$ on $\D$ is the zero divisor of a holomorphic
function on $\D$ with Taylor coefficients in $\Z$ if and only if $D$
is invariant under complex conjugation.
\end{theorem}

The necessity of conjugation invariance is immediate; sufficiency is
the content.  Conceptually, the result may be viewed as an
interpolation theorem inside the coefficient-restricted algebra
$\Z[[z]] \cap \OD$: every divisor class has an arithmetic
representative.  The construction is non-canonical: different orderings
of the zeros and different rounding choices yield different realizing
functions, and the coefficient sequences can vary substantially with
these choices.

The technical heart of the paper is the following result, from which
Theorem~\ref{thm:main} is deduced.

\begin{theorem}[Gaussian-integer realization]
\label{prop:Zi}
Every effective divisor on $\D$ is the zero divisor of a holomorphic
function on $\D$ with Taylor coefficients in the Gaussian integers
$\Zi$.
\end{theorem}

Theorem~\ref{thm:main} follows from Theorem~\ref{prop:Zi} by a
conjugate-pairing argument: given a conjugation-stable effective
divisor $D$, apply Theorem~\ref{prop:Zi} to obtain $g$ with zero
divisor $D$ and $\Zi$-coefficients, then pair each complex zero with
its conjugate to force all coefficients into $\Z$
(Section~\ref{sec:proofs}).

The construction common to both results deforms each classical
Weierstrass elementary factor by a single free parameter, governed by
a simple principle: each modified factor of order $n$ leaves all
Taylor coefficients of degree $\leq n$ unchanged, while shifting the
coefficient of degree $n+1$ by a controlled affine amount.  This
triangular structure allows coefficients to be rounded to integers
one at a time without disturbing previously corrected coefficients and
without affecting the convergence of the product.

\medskip\noindent\textbf{Proof architecture.}
Section~\ref{sec:factors} defines the modified factors and establishes
their two key properties: exact coefficient-vanishing below order $n$
(Lemma~\ref{lem:structure}), and affine control of the coefficient at
order $n+1$ (Lemma~\ref{lem:affine}).  Section~\ref{sec:proofs} proves
Theorem~\ref{prop:Zi} by an inductive rounding argument, then deduces
Theorem~\ref{thm:main} by the conjugate-pairing device.
Section~\ref{sec:example} illustrates the first two steps of the
induction concretely.  Section~\ref{sec:algebra} derives the
ring-theoretic consequences.

\medskip

The question of which holomorphic functions have integer Taylor
coefficients has a classical literature.  The P\'olya--Carlson theorem
\cite{PC} characterizes rational functions among power series with
integer coefficients and radius of convergence~$1$.  Results on
integer-valued entire functions \cite{Gel} identify growth conditions
that force rationality.  The prescribed-zero-set problem itself is
classical without arithmetic constraints: Blaschke products realize
any Blaschke sequence as a zero set of a bounded holomorphic function
on $\D$, and canonical products \cite{Conway,Levin} handle the general
discrete case.  Our result combines both perspectives: we show that
prescribed zero sets can always be realized by integer-coefficient
functions, with no arithmetic constraint beyond conjugation invariance.

The closest conceptual antecedents are Borel's interpolation theorem
(every sequence of values can be prescribed along a discrete sequence
for a smooth function, or for an entire function on the whole plane)
and universal Taylor series (entire functions whose partial sums
approximate arbitrary functions on arbitrary compact sets).  Neither
framework imposes arithmetic constraints on coefficients, and neither
directly addresses the convergence-on-a-disk setting.  For rings of
holomorphic functions with arithmetic coefficient conditions, see
\cite{Stout} for background on analytic function algebras.  We are not aware of a prior result combining prescribed divisors with
integrality of Taylor coefficients, though our literature search was
not exhaustive.

The same construction, with zeros placed outside $\overline{\D}$,
produces holomorphic functions on $\D$ with Gaussian-integer Taylor
coefficients and no zeros at all.  Section~\ref{sec:algebra} uses this
and Theorem~\ref{prop:Zi} to derive ring-theoretic consequences for
the subrings $\mR = \Zi[[z]] \cap \OD$ and $\mRR = \Z[[z]] \cap \OD$.

\section{Modified Elementary Factors}
\label{sec:factors}

\begin{definition}
For $n \geq 0$ and $c \in \C$, the \emph{modified elementary factor}
of order $n$ with parameter $c$ is
\[
  E_n(w;\,c)
  = (1-w)\exp\!\Bigl(\sum_{k=1}^{n}\frac{w^k}{k} + \frac{c\,w^{n+1}}{n+1}\Bigr).
\]
\end{definition}

For $c = 1$ this is the classical Weierstrass factor $E_n(w)$.  The
crucial properties are collected in the following lemma.

\begin{lemma}[Structure of the modified factor]
\label{lem:structure}
For all $n \geq 0$, $c \in \C$, and $w \in \D$,
\begin{equation}
\label{eq:Eexp}
  E_n(w;\,c) = \exp\bigl(G_n(w;\,c)\bigr),
\end{equation}
where
\begin{equation}
\label{eq:Gdef}
  G_n(w;\,c)
  = \frac{(c-1)\,w^{n+1}}{n+1} - \sum_{k=n+2}^{\infty}\frac{w^k}{k}.
\end{equation}
In particular:
\begin{enumerate}
\item[(i)] $E_n(0;\,c) = 1$.
\item[(ii)] $[w^m]\,E_n(\,\cdot\,;\,c) = 0$ for $1 \leq m \leq n$,
  independently of $c$.
\item[(iii)] $[w^{n+1}]\,E_n(\,\cdot\,;\,c) = (c-1)/(n+1)$,
  linear in $c$ with slope $1/(n+1)$.
\item[(iv)] $E_n(w;\,c) \neq 0$ for all $w \in \D$.
\item[(v)] Although \eqref{eq:Eexp} holds only on $\D$, the original
  definition extends to an entire function of $w \in \C$, which has a
  simple zero at $w = 1$ and no other zeros.
\end{enumerate}
\end{lemma}

\begin{proof}
Since $-\log(1-w) = \sum_{k=1}^{\infty} w^k/k$ for $w \in \D$, the
exponent in the definition equals
\[
  -\log(1-w) + \frac{(c-1)w^{n+1}}{n+1} - \sum_{k=n+2}^{\infty}\frac{w^k}{k},
\]
so $(1-w)\cdot(1-w)^{-1}\cdot\exp(G_n(w;\,c)) = \exp(G_n(w;\,c))$
for $w \in \D$, giving \eqref{eq:Eexp}--\eqref{eq:Gdef}.

(i): $G_n(0;\,c) = 0$.

(ii)--(iii): Since $G_n(w;\,c) = \tfrac{c-1}{n+1}w^{n+1} + O(w^{n+2})$,
we have $e^{G_n} = 1 + \tfrac{c-1}{n+1}w^{n+1} + O(w^{n+2})$,
giving (ii) and (iii) directly.

(iv): $e^{G_n}$ is nowhere zero on $\D$.

(v): From the definition, $E_n(w;\,c) = (1-w)\,e^{H(w)}$ where
$H(w) = \sum_{k=1}^n w^k/k + cw^{n+1}/(n+1)$ is an entire function;
since $e^{H(w)}$ is nowhere zero, the zeros of $E_n(\,\cdot\,;\,c)$ are
exactly those of $(1-w)$, i.e., the simple zero at $w=1$.
\end{proof}

\begin{remark}
Parts~(iv) and~(v) use different representations of $E_n$: (iv) uses
\eqref{eq:Eexp}, valid for $w \in \D$; (v) uses the original definition,
valid for all $w \in \C$.  They are consistent: as $w \to 1^-$ along
the real axis, $G_n(w;\,c) \to -\infty$, so $e^{G_n} \to 0$.
\end{remark}

The following lemma makes the triangular coefficient structure precise.

\begin{lemma}[Affine coefficient control]
\label{lem:affine}
Let $a \in \C \setminus \{0\}$, $n \geq 0$, and $h(z) =
\sum_{m \geq 0} h_m z^m$ a power series convergent near $0$ with
$h_0 = 1$.  For $F(z) = h(z)\cdot E_n(z/a;\,c)$:
\begin{enumerate}
\item[(i)] $[z^m]\,F = h_m$ exactly, for $0 \leq m \leq n$,
  independently of $c$; in particular, introducing the factor
  $E_n(z/a;\,c)$ leaves all coefficients of degree $\leq n$ unchanged.
\item[(ii)] $[z^{n+1}]\,F = h_{n+1} + \dfrac{c-1}{(n+1)a^{n+1}}$,
  affine in $c$ with nonzero slope $1/((n+1)a^{n+1})$.
\end{enumerate}
\end{lemma}

\begin{proof}
Write $E_n(z/a;\,c) = \sum_{j \geq 0} e_j z^j$.  By
Lemma~\ref{lem:structure}: $e_0 = 1$, $e_j = 0$ for $1 \leq j \leq n$,
$e_{n+1} = (c-1)/((n+1)a^{n+1})$.  Then $[z^m]F = \sum_{j=0}^m
h_{m-j}\,e_j$.

(i) For $m \leq n$: the vanishing $e_1 = \cdots = e_n = 0$ means only
$j=0$ contributes, giving $h_m$.

(ii) For $m=n+1$: terms $j=1,\ldots,n$ vanish, leaving
$h_{n+1}e_0 + h_0 e_{n+1} = h_{n+1} + (c-1)/((n+1)a^{n+1})$.
\end{proof}

\begin{remark}[Triangular-affine structure]
\label{rem:triangular}
Lemma~\ref{lem:affine} says that the coefficient map
\[
  c \mapsto \bigl([z^0](h\cdot E_n(z/a;\,c)),\,
                  [z^1](h\cdot E_n(z/a;\,c)),\,\ldots\bigr)
\]
is \emph{triangular with surjective affine diagonal}: the first $n+1$
entries $(m = 0,\ldots,n)$ are independent of $c$, and the entry at
$m = n+1$ is an affine surjection onto $\C$.  In this sense each $E_n(z/a;\,c)$ operates as a one-coefficient
corrector: it fixes one new Taylor coefficient without disturbing any
previously corrected one.  The resulting induction resembles Borel-type
interpolation or
triangular normal-form constructions in formal algebra, here carried
out in a convergent analytic product.
\end{remark}

\section{Proofs}
\label{sec:proofs}

\begin{proof}[Proof of Theorem~\ref{prop:Zi}]
Let $D$ have support $S$ with multiplicities $\{m_a\}_{a \in S}$.
Write $m_0 = m_0(D)$ for the multiplicity at $0$ (possibly zero).
Factor out the zero at the origin: if $m_0 > 0$, write $D = m_0[0] + D'$
where $D'$ has support in $S \setminus \{0\}$.  It suffices to realize
$D'$ by a function $g \in \OD$ with $\Zi$-coefficients and $g(0) \neq 0$
(produced by the construction below), and then set $f(z) = z^{m_0}g(z)$,
which has $\Zi$-coefficients and zero divisor $D$.

Henceforth assume $0 \notin S$.  If $S$ is finite (or empty), the
construction below terminates after finitely many factors: $P = \prod_{n}
E_n(z/a_{n+1};\,c_{n+1})$ is then a finite product of entire functions,
the convergence step is vacuous, and the inductive forcing realizes all
coefficients in $\Zi$ in finitely many stages.  We therefore assume $S$
is infinite, so that $A = \{a_n\}_{n=1}^\infty$ is an infinite multiset.
Let $A$ be the
multiset $S$ with multiplicities, enumerated with $\abs{a_n} \leq
\abs{a_{n+1}}$.  Such an enumeration exists and satisfies $\abs{a_n} \to 1$
because $S$ is discrete: every closed disk $\abs{z} \leq r \subset \D$
contains only finitely many points of $S$ (each with finite
multiplicity), so only finitely many terms of the sequence lie in
$\abs{z} \leq r$.  In particular $a_n \neq 0$ for all $n$.
Since $A$ is discrete in $\D$, each point of $S$ occurs with finite
multiplicity.  The construction depends on the enumeration: different
orderings and rounding choices yield different functions $f$ with the
same zero divisor.  The function produced is not canonical.
Set $P_0 = 1$ and
\[
  P_{N}(z) = \prod_{n=0}^{N-1}
  E_n\!\Bigl(\frac{z}{a_{n+1}};\,c_{n+1}\Bigr), \quad N \geq 1,
\]
with parameters $c_n \in \C$ chosen inductively.

\medskip\noindent\textbf{Step 1: Inductive coefficient forcing.}

\begin{claim}
For each $N \geq 0$, parameters $c_1,\ldots,c_N$ can be chosen so that
\begin{equation}
\label{eq:indhyp}
  [z^m]\,P_N \in \Zi \quad (0 \leq m \leq N),
\end{equation}
with the bounds
\begin{equation}
\label{eq:cnbound}
  \abs{c_{n} - 1} \leq \tfrac{\sqrt{2}}{2}\,n\,\abs{a_{n}}^{n},
  \quad 1 \leq n \leq N.
\end{equation}
\end{claim}

\begin{proof}[Proof of Claim]
Induction on $N$.  The base case $N=0$ holds since $[z^0]P_0 = 1 \in
\Zi$.  Given the hypothesis at $N$, apply Lemma~\ref{lem:affine} with
$h = P_N$, $a = a_{N+1}$, $n = N$:
\begin{itemize}
  \item For $0 \leq m \leq N$: $[z^m]P_{N+1} = [z^m]P_N \in \Zi$
        (induction hypothesis and Lemma~\ref{lem:affine}(i); later
        factors affect only degrees $> N$, so this coefficient is now
        frozen).
  \item For $m = N+1$:
    $[z^{N+1}]P_{N+1} = [z^{N+1}]P_N + \dfrac{c_{N+1}-1}{(N+1)a_{N+1}^{N+1}}$
    (Lemma~\ref{lem:affine}(ii)).
\end{itemize}
The second line is affine in $c_{N+1}$ with nonzero slope.  Let $T \in
\Zi$ be any nearest Gaussian integer to $[z^{N+1}]P_N$ (rounding real
and imaginary parts independently; if there is more than one nearest
Gaussian integer, choose any one) and set
$c_{N+1}$ to achieve $[z^{N+1}]P_{N+1} = T \in \Zi$.  Then
\[
  \abs{c_{N+1}-1}
  = \abs{T - [z^{N+1}]P_N} \cdot (N+1)\abs{a_{N+1}}^{N+1}
  \leq \tfrac{\sqrt{2}}{2}(N+1)\abs{a_{N+1}}^{N+1},
\]
since the nearest-$\Zi$ rounding error satisfies $\abs{T-v} \leq
\sqrt{2}/2$.
\end{proof}

\medskip\noindent\textbf{Step 2: Freezing of coefficients.}

By Lemma~\ref{lem:affine}(i), once the factor $E_m$ has been introduced
at stage $m$, all subsequent factors contribute only in degrees $> m$.
Hence $[z^m]P_N$ is constant for all $N > m$, and the coefficient of
degree $m$ in the limit $f$ equals $[z^m]P_{m+1} \in \Zi$.

\medskip\noindent\textbf{Step 3: Convergence.}

Fix $0 < r < 1$ and choose $s$ with $r < s < 1$.  Since
$\abs{a_n} \to 1$, there exists $N_0$ such that $\abs{a_{n+1}} \geq s$
for all $n \geq N_0$; then $r/\abs{a_{n+1}} \leq r/s < 1$ for such $n$.
By \eqref{eq:Eexp}--\eqref{eq:Gdef} and \eqref{eq:cnbound}, for $n \geq
N_0$ and $\abs{z} \leq r$:
\begin{align*}
  \abs{G_n\!\bigl(z/a_{n+1};\,c_{n+1}\bigr)}
  &\leq \frac{\abs{c_{n+1}-1}}{n+1}\Bigl(\frac{r}{\abs{a_{n+1}}}\Bigr)^{n+1}
    + \sum_{k \geq n+2} \Bigl(\frac{r}{\abs{a_{n+1}}}\Bigr)^k \\
  &\leq \frac{\sqrt{2}}{2}\,r^{n+1}
    + \frac{(r/s)^{n+2}}{1-r/s}.
\end{align*}
The first term uses $\abs{c_{n+1}-1}/(n+1) \leq
\frac{\sqrt{2}}{2}\abs{a_{n+1}}^{n+1}$ from \eqref{eq:cnbound}, after
which the factor $\abs{a_{n+1}}^{n+1}$ cancels
$(r/\abs{a_{n+1}})^{n+1}$ down to $r^{n+1}$, independently of whether
$\abs{a_{n+1}} \lessgtr 1$.  The second uses $r/\abs{a_{n+1}} \leq r/s$
to sum the geometric tail; note that for zeros inside $\D$ one has
$\abs{a_{n+1}} < 1$, so $r/\abs{a_{n+1}} > r$ and the cruder bound
$r^{n+2}/(1-r)$ is \emph{not} available; the comparison must be made
against $s \leq \abs{a_{n+1}}$.  In particular
$\abs{G_n} = O\bigl((r/s)^{n}\bigr)$ uniformly on $\abs{z} \leq r$, so
$\sum_{n \geq N_0} \sup_{\abs{z}\leq r} \abs{G_n} < \infty$.  For
sufficiently large $n$,
$\sup_{\abs{z}\leq r}\abs{G_n} < 1/2$, and then $\abs{e^w - 1} \leq
2\abs{w}$ for $\abs{w} \leq 1/2$ gives $\abs{E_n - 1} = \abs{e^{G_n}
- 1} \leq 2\abs{G_n}$.  Since the tail factors are holomorphic and
nonzero on $\abs{z} \leq r$ (as $\abs{a_{n+1}} > r$), the standard
infinite-product criterion (if $f_n$ are holomorphic, nonzero on
compact $K$, and $\sum \sup_K \abs{f_n - 1} < \infty$, then $\prod
f_n$ converges uniformly on $K$ to a holomorphic nonzero limit; see
\cite[Ch.~5, Thm.~1.1]{Conway}) applies.  Since $r < 1$ was arbitrary,
$f$ is holomorphic on $\D$.

\medskip\noindent\textbf{Step 4: Zero set.}

By Lemma~\ref{lem:structure}(v), each factor $E_n(z/a_{n+1};\,c_{n+1})$
has a simple zero at $z = a_{n+1}$ and no other zeros.  Fix any point
$p \in \D$.  Since $A$ is discrete, $p$ appears only finitely many
times in the sequence, say at positions $j_1 < \cdots < j_\ell$ (so
$a_{j_t} = p$; here $\ell \geq 0$, and $\ell = 0$ if $p \notin A$).
The factor of index $n$ vanishes at $p$ exactly when $a_{n+1} = p$,
i.e.\ when $n + 1 \in \{j_1,\ldots,j_\ell\}$; the largest such index is
$n = j_\ell - 1$.  Hence for all $n \geq j_\ell$ the factor
$E_n(z/a_{n+1};\,c_{n+1})$ is nonzero at $p$ (its only zero is at
$a_{n+1} \neq p$).  Choose
$\varepsilon > 0$ small enough that the disk $\abs{z - p} \leq
\varepsilon$ contains no $a_{n+1}$ with $n \geq j_\ell$; then all tail
factors are holomorphic and nonzero on this neighborhood.
$\sum_{n \geq j_\ell} \sup_{|z-p| \leq \varepsilon} |E_n - 1| < \infty$
for small $\varepsilon > 0$; by the infinite-product criterion already
invoked in Step~3, the tail product converges uniformly near $p$ to a
holomorphic function nonvanishing at $p$.  The initial segment
$\prod_{n < j_\ell}$ vanishes at $p$ to order $\ell$ (each of the
$\ell$ factors with $a_{n+1} = p$ contributes a simple zero there).
Hence $f$ has a zero of order $\ell$ at $p$, and $\ell = 0$ gives
nonvanishing when $p \notin A$.  The constant term is
$f(0) = \prod_n 1 = 1$.\end{proof}

\begin{proof}[Proof of Theorem~\ref{thm:main}]
Necessity was established in the introduction.  For sufficiency, let
$D = \sum_{a \in S} m_a[a]$ be an effective divisor on $\D$ with
discrete support $S$ invariant under conjugation ($m_{\bar a} = m_a$
for all $a \in S$).  Enumerate the support with multiplicities as a
sequence $A = \{a_n\}$ where each $a \in S$ appears $m_a$ times.
Reorder so that $\abs{a_n} \leq \abs{a_{n+1}}$, and within each
modulus group arrange so that each nonreal $a$ is immediately followed
by its conjugate $\bar a$ (possible since $m_a = m_{\bar a}$).

Apply the construction of Theorem~\ref{prop:Zi} with $A$, grouping the
factors into two types designed to keep partial-product coefficients
real.  For each index $n \geq 0$, define:
\begin{itemize}
  \item \emph{Real zero}: a single $b \in A \cap \R$; use
        $F_n(z) = E_n(z/b;\,c_n)$ with $c_n \in \R$.
  \item \emph{Conjugate pair}: $\{a, \bar{a}\} \subset A$,
        $a \notin \R$; use
        $F_n(z) = E_n(z/a;\,c_n)\cdot E_n(z/\bar{a};\,\bar{c}_n)$.
\end{itemize}
In the conjugate-pair case the two factors are complex conjugates of each
other as functions of $z$, so their product has real Taylor coefficients.
Since each factor expands as $1 + \alpha z^{n+1} + O(z^{n+2})$ and
$1 + \bar\alpha z^{n+1} + O(z^{n+2})$ with
$\alpha = (c_n-1)/((n+1)a^{n+1})$, their product is
\[
  1 + (\alpha + \bar\alpha)z^{n+1} + O(z^{n+2})
  = 1 + 2\operatorname{Re}(\alpha)\,z^{n+1} + O(z^{n+2}),
\]
since the cross term $\alpha\bar\alpha z^{2n+2}$ appears only at degree
$2n+2 \geq n+2$.

In both slot types, Lemma~\ref{lem:affine}(i) gives exact non-interference
at degrees $0,\ldots,n$.  Each conjugate-pair slot consumes a single
inductive stage $n$, handling two zeros simultaneously; the indexing
of the coefficient-freezing argument is unchanged because both factors
in the pair have the same order $n$ and jointly contribute only at
degree $n+1$.  For the degree-$(n+1)$ adjustment: in the real-zero
slot it is $(c_n-1)/((n+1)b^{n+1}) \in \R$ (since $b, c_n \in \R$),
a nonzero real-linear function of $c_n$, hence surjective onto $\R$.
In the conjugate-pair slot, write $c_n = x + iy$; the adjustment is
\begin{align*}
  2\operatorname{Re}(\alpha)
  &= \frac{2\operatorname{Re}\bigl((c_n-1)\bar{a}^{n+1}\bigr)}{(n+1)\abs{a}^{2(n+1)}} \\
  &= \frac{2}{(n+1)\abs{a}^{2(n+1)}}
    \bigl(\operatorname{Re}(\bar{a}^{n+1})\,(x-1)
          - \operatorname{Im}(\bar{a}^{n+1})\,y\bigr),
\end{align*}
a nonconstant real-affine map $\C \to \R$, hence surjective onto $\R$
and in particular taking every integer value.  One chooses $c_n$ to
round the degree-$(n+1)$ coefficient to the nearest integer in $\Z$.

Because this affine map has a whole line of preimages, $c_n$ is not
determined; we fix it to be the preimage of minimal modulus, namely the
one along the gradient direction.  The map
$c_n - 1 \mapsto 2\operatorname{Re}(\alpha)$ has gradient of norm
$\frac{2}{(n+1)\abs{a}^{2(n+1)}}\abs{\bar a^{n+1}}
= \frac{2}{(n+1)\abs{a}^{n+1}}$, and the rounding target has modulus at
most $\tfrac12$, so this choice gives
\[
  \abs{c_n - 1}
  \;\leq\; \frac{1}{2}\cdot\frac{(n+1)\abs{a}^{n+1}}{2}
  \;=\; \frac{1}{4}(n+1)\abs{a}^{n+1},
\]
and likewise $\abs{\bar c_n - 1} = \abs{c_n - 1}$.  Thus both factors of
the pair satisfy the bound \eqref{eq:cnbound} (with room to spare), and
the real-zero slots satisfy it as well, with rounding error $\leq
\tfrac12$.  The Step~3 estimate therefore applies verbatim to each
individual factor; a conjugate-pair slot contributes two such factors of
common order $n$, at most doubling the per-stage bound, so
$\sum_n \sup_{\abs{z}\leq r}\abs{G_n} < \infty$ still holds and the
product converges on $\D$.  The zero-set argument is unchanged.

The resulting function has zero set exactly $A$ and Taylor coefficients in $\Z$.
\end{proof}

\begin{remark}[Alternative symmetrization approach]
An alternative route to Theorem~\ref{thm:main} would be to apply
Theorem~\ref{prop:Zi} to realize $D$ by some $g \in \mR$, then form
$f(z) = g(z)\,\overline{g(\bar z)}$.  The function $\overline{g(\bar z)}$
has Taylor coefficients $\overline{b_n}$ (conjugates of those of $g$),
and the product has coefficients $\sum_k b_k \overline{b_{n-k}} \in \R$
when $D$ is conjugation-stable, but these are not generally integers.
The direct slot-factor argument avoids this issue by maintaining real
integer coefficients throughout the induction, without passing through a
$\Zi$-coefficient intermediate.
\end{remark}

\begin{proposition}[Nowhere-vanishing $\Zi$-coefficient functions]
\label{prop:nv}
Let $\{a_n\}_{n=1}^{\infty} \subset \C \setminus \overline{\D}$ with
$\abs{a_n} \to 1$.  There exists $f \in \OD$ with Taylor coefficients
in $\Zi$, nowhere vanishing on $\D$, with $f(0) = 1$.
\end{proposition}

\begin{proof}
The proof of Theorem~\ref{prop:Zi} applies without change.  Since
each $a_{n+1} \notin \D$, the zero of each factor lies outside $\D$,
so $f$ is nowhere vanishing on $\D$.  For large $n$,
$r/\abs{a_{n+1}} < r < 1$ (since $\abs{a_n} \to 1$ from above), and
the convergence estimate is unchanged.
\end{proof}

\section{A Worked Example}
\label{sec:example}

We illustrate the inductive rounding step.  Take $a_1 = \frac{1}{3} +
\frac{i}{4} \in \D$.

\medskip\noindent\textbf{Step $N=0$.}  We form $P_1 = E_0(z/a_1;\,c_1)$.
By Lemma~\ref{lem:affine}(ii), $[z^1]P_1 = (c_1-1)/a_1$.  The nearest
$\Zi$-element to $[z^1]P_0 = 0$ is $T_1 = 0$, so $c_1 = 1$ (the classical
factor) and $[z^1]P_1 = 0 \in \Zi$.

\medskip\noindent\textbf{Step $N=1$.}  With $c_1 = 1$, the factor
$E_0(z/a_1;\,1) = (1-z/a_1)e^{z/a_1}$ gives $[z^2]P_1 = -1/(2a_1^2)
\approx -0.806 + 2.765i$, which is not a Gaussian integer.
The nearest element of $\Zi$ is $T_2 = -1+3i$; the rounding error
$\abs{T_2 - [z^2]P_1} = \sqrt{58}/25 \approx 0.305$ satisfies the
bound $\leq \sqrt{2}/2$.  A second zero $a_2$ introduces the factor
$E_1(z/a_2;\,c_2)$, which shifts $[z^2]$ by $(c_2-1)/(2a_2^2)$.
We choose $c_2$ to achieve $[z^2]P_2 = -1+3i \in \Zi$.

After these two steps: $[z^0]P_2 = 1$, $[z^1]P_2 = 0$,
$[z^2]P_2 = -1+3i$, all in $\Zi$.  At step~0 the classical choice
$c_1 = 1$ sufficed; at step~1 the classical product gave a non-integer
coefficient and nontrivial rounding was required.

\section{Ring-Theoretic Consequences}
\label{sec:algebra}

Let $\mR = \Zi[[z]] \cap \OD$ and $\mRR = \Z[[z]] \cap \OD$, and write
$\mathfrak{n}_0 = \{g \in \mR : g(0) = 0\}$ for the augmentation ideal.
We derive only what follows directly from the main results, leaving the
deeper algebraic structure of $\mR$ for Section~\ref{sec:open}.

We first record that the units of $\mR$ are completely determined: a
nowhere-vanishing element with admissible constant term is automatically
invertible, so the construction of Proposition~\ref{prop:nv} in fact
produces units.

\begin{proposition}[Units of $\mR$]
\label{prop:units}
An element $f \in \mR$ is a unit if and only if $f(0) \in \Zi^\times =
\{\pm 1, \pm i\}$ and $f$ is nowhere vanishing on $\D$.  The analogous
statement holds for $\mRR$, with $\Zi^\times$ replaced by
$\Z^\times = \{\pm 1\}$.
\end{proposition}

\begin{proof}
($\Rightarrow$) If $fg = 1$ with $g \in \mR$ then $f(0)g(0) = 1$ in
$\Zi$, so $f(0) \in \Zi^\times$; and $f(z)g(z) = 1$ for all $z \in \D$
forces $f$ nowhere-vanishing on $\D$.

($\Leftarrow$) Suppose $u := f(0) \in \Zi^\times$ and $f$ has no zero in
$\D$.  Writing $f = u(1 + a_1 z + a_2 z^2 + \cdots)$ with $a_k =
u^{-1}[z^k]f \in \Zi$ (as $u^{-1} \in \Zi$), the constant term $1$ is a
unit of the ring $\Zi[[z]]$, so $f$ has a formal inverse $f^{-1} \in
\Zi[[z]]$ (its coefficients are computed by the usual triangular
recurrence, which stays in $\Zi$).  On the other hand, since $f$ is
holomorphic and nowhere zero on $\D$, the analytic reciprocal $1/f$ is
holomorphic on $\D$, i.e.\ $1/f \in \OD$.  Its Taylor expansion at $0$
coincides with the formal inverse $f^{-1}$ by uniqueness of power-series
expansion, so $f^{-1} \in \Zi[[z]] \cap \OD = \mR$.  Hence $f$ is a unit.
The argument for $\mRR$ is identical with $\Z$ in place of $\Zi$.
\end{proof}

\begin{remark}
\label{rem:nv-units}
In particular every function produced by Proposition~\ref{prop:nv} (all
zeros placed outside $\overline{\D}$, $f(0) = 1$) is a unit of $\mR$, so
$\mR^\times$ contains transcendental units with prescribed zero sets in
$\C \setminus \overline{\D}$ and is far larger than $\Zi^\times$.
There is no nowhere-vanishing $f \in \mR$ with $f(0) \in \Zi^\times$
that fails to be a unit; an apparent counterexample such as $1 + 2z$
($f(0) = 1$, formal inverse $\sum(-2)^n z^n \notin \OD$) is excluded
precisely because it \emph{does} vanish in $\D$, at $z = -\tfrac12$.
Note also that no \emph{polynomial} in $\mRR$ with $f(0) = \pm 1$ can be
nowhere-vanishing on $\D$ without being a unit: such an $f$ would have
all roots in $\abs{z} \geq 1$, and comparing the constant and leading
coefficients via $\prod \abs{\rho_j} = 1/\abs{c_d} \leq 1$ forces every
root onto $\partial\D$, whence (Kronecker) the roots are roots of unity
and $1/f \in \mRR$.  The genuinely interesting units are therefore the
transcendental ones.
\end{remark}

The key algebraic consequence of Theorem~\ref{prop:Zi} is that every
element of $\OD$ is divisible by an element of $\mR$ up to a unit.
Corollary~\ref{cor:ideals} and Proposition~\ref{prop:inject} flow directly
from this: the former says every ideal of $\OD$ is generated by its
$\mR$-elements, and the latter uses this to show the contraction map is
injective.

\begin{proposition}[Factorization up to units in $\OD$]
\label{prop:associate}
Every $f \in \OD$ can be written $f = g \cdot u$ with $g \in \mR$ and
$u \in \OD^\times$.  Here $f$ and $g$ generate the same principal ideal
in $\OD$.  The unit $u$ need not lie in $\mR$, so this is not a
factorization within $\mR$; in particular it does not imply that $\mR$
is divisor-theoretically rich.
\end{proposition}

\begin{proof}
Apply Theorem~\ref{prop:Zi} to the zero divisor of $f$ to obtain
$g \in \mR$ with the same zeros and multiplicities.  Since $f$ and $g$
have the same zeros with the same multiplicities, the ratio $u = f/g$
extends to a holomorphic function on $\D$ with no zeros, i.e.,
$u \in \OD^\times$.
\end{proof}

\begin{corollary}
\label{cor:ideals}
Every ideal $I \subset \OD$ satisfies $I = (I \cap \mR)\cdot\OD$.
\end{corollary}

\begin{proof}
For any $f \in I$, write $f = gu$ with $g \in \mR$, $u \in \OD^\times$.
Then $g = fu^{-1} \in I$, so $g \in I \cap \mR$ and
$f = gu \in (I \cap \mR)\cdot\OD$.
\end{proof}

We now compare $\mathrm{MaxSpec}(\OD)$ and $\mathrm{Spec}(\mR)$.  We
make no use of any classification of maximal ideals of $\OD$.  The
inclusion $\mR \hookrightarrow \OD$ induces a contraction map
\[
  \phi \colon \mathrm{MaxSpec}(\OD) \to \mathrm{Spec}(\mR),
  \qquad \mathfrak{m} \mapsto \mathfrak{m} \cap \mR.
\]
Note that the contraction of a maximal ideal of $\OD$ need not be
maximal in $\mR$; $\phi$ maps into $\mathrm{Spec}(\mR)$, not
$\mathrm{MaxSpec}(\mR)$.  The map is well-defined: $1 \notin \mathfrak{m}
\cap \mR$ (since $1 \notin \mathfrak{m}$), and $\mathfrak{m} \cap \mR
\neq 0$ because any nonzero $f \in \mathfrak{m}$ factors as $f = gu$
with $g \in \mR$, $u \in \OD^\times$, giving $g = fu^{-1} \in
\mathfrak{m} \cap \mR$.

\begin{proposition}
\label{prop:inject}
The map $\phi$ is injective.
\end{proposition}

\begin{proof}
Suppose $\mathfrak{m} \cap \mR = \mathfrak{m}' \cap \mR$.  For any
$f \in \mathfrak{m}$, write $f = gu$ with $g \in \mR$,
$u \in \OD^\times$.  Then $g \in \mathfrak{m} \cap \mR =
\mathfrak{m}' \cap \mR \subseteq \mathfrak{m}'$, so $f = gu \in
\mathfrak{m}'$.  Thus $\mathfrak{m} \subseteq \mathfrak{m}'$, and by
maximality $\mathfrak{m} = \mathfrak{m}'$.
\end{proof}

\begin{proposition}
\label{prop:fiber}
The augmentation ideal satisfies $\mR/\mathfrak{n}_0 \cong \Zi$.  The
primes of $\mR$ containing $\mathfrak{n}_0$ biject with
$\mathrm{Spec}(\Zi)$.  Among them, only $\mathfrak{n}_0$ itself lies
in $\mathrm{im}(\phi)$ (as $\phi(\mathfrak{m}_0) = \mathfrak{n}_0$
where $\mathfrak{m}_0 = \ker(\mathrm{ev}_0) \subset \OD$); no prime
of $\mR$ strictly containing $\mathfrak{n}_0$ lies in $\mathrm{im}(\phi)$.
\end{proposition}

\begin{proof}
Evaluation at $0$ gives a surjection $\mathrm{ev}_0\colon\mR \to \Zi$
with kernel $\mathfrak{n}_0$, so $\mR/\mathfrak{n}_0 \cong \Zi$.  The
bijection with $\mathrm{Spec}(\Zi)$ is the standard quotient correspondence.

For the last claim: if $\phi(\mathfrak{m}) \supseteq \mathfrak{n}_0$
then $\mathfrak{n}_0 \cdot \OD \subseteq \mathfrak{m}$.  But
$\mathfrak{n}_0 \cdot \OD = \{f \in \OD : f(0)=0\} = \mathfrak{m}_0$
is maximal, so $\mathfrak{m} = \mathfrak{m}_0$ and
$\phi(\mathfrak{m}) = \mathfrak{n}_0$, not any strictly larger ideal.
\end{proof}

\begin{remark}[The arithmetic fiber]
\label{rem:fiber}
Propositions~\ref{prop:inject} and~\ref{prop:fiber} together say that
$\mathrm{MaxSpec}(\OD)$ embeds injectively into $\mathrm{Spec}(\mR)$,
with image disjoint from the primes strictly above $\mathfrak{n}_0$.
Those primes (the maximal ideals of $\mR/\mathfrak{n}_0 \cong \Zi$,
i.e., the principal ideals $(1+i)$, $(2+i)$, $(3)$, $(2+3i)$, \ldots)
are arithmetic data present in $\mathrm{Spec}(\mR)$ but absent from
$\mathrm{MaxSpec}(\OD)$.

The analogous statements hold for $\mRR$ with $\Z$ in place of $\Zi$.
The natural map $\mathrm{Spec}(\Zi) \to \mathrm{Spec}(\Z)$ over the
origin is the classical one: $p \equiv 1 \pmod{4}$ splits in $\Zi$,
$p \equiv 3 \pmod{4}$ remains inert, and $p=2$ ramifies.
\end{remark}

\section{Further Questions}
\label{sec:open}

\begin{itemize}
\item \textbf{Group structure of $\mR^\times$.}
  Proposition~\ref{prop:units} identifies $\mR^\times$ as the set of
  nowhere-vanishing elements with constant term in $\Zi^\times$, and
  Proposition~\ref{prop:nv} shows this group is large.  The map
  $f \mapsto (f(0),\, f/f(0))$ splits $\mR^\times$ as
  $\Zi^\times \times U_1$, where $U_1 = \{f \in \mR^\times : f(0) = 1\}$
  is the group of nowhere-vanishing $\Zi$-series normalized at the
  origin.  Describing $U_1$ intrinsically (for instance, whether
  $\log f$ admits a useful normal form, or how $U_1$ sits inside the
  multiplicative group of nowhere-vanishing elements of $\OD$) remains
  open.

\item \textbf{Surjectivity of $\phi$.}  Whether every maximal ideal of
  $\mR$ not containing $\mathfrak{n}_0$ arises by contraction from
  $\OD$ is open.  Surjectivity would require showing that
  $\mathfrak{n}\cdot\OD$ is proper for every such maximal
  $\mathfrak{n} \subset \mR$.

\item \textbf{Algebraic structure of $\mR$.}  We are unaware of results
  concerning whether $\mR$ is coherent, integrally closed, or B\'ezout;
  see \cite{Stout} for background on rings of holomorphic functions.

\item \textbf{Coefficient growth and enumeration dependence.}
  The inductive construction shows that $\abs{[z^n]f}$ is controlled
  by the products of earlier partial-product values and the rounding
  parameters $c_k$, but the resulting coefficients can depend
  sensitively on the ordering of the zeros.  Concrete upper bounds on
  coefficient growth in terms of the zero distribution, and the
  question of whether the ordering can be chosen to minimize coefficient
  growth or force bounded coefficients, are natural problems that the
  construction raises but does not address.

\item \textbf{Growth-class variants.}  The construction extends in
  principle to subrings $\mathcal{A}_\omega = \{f \in \OD :
  \log M(r,f) = O(\omega(r))\}$ for suitable $\omega$.  Precise
  control of the growth of the modified product relative to the
  classical canonical product, and the associated ring-theoretic
  consequences, remain to be worked out.
\end{itemize}

\section*{Tool and Computational Resource Disclosure}

This manuscript was drafted and revised with the assistance of
Claude (Anthropic; primarily Claude Opus~4.6), a large language model used
interactively through Claude Code.  Claude contributed to composing
and revising portions of the exposition, and produced the
accompanying Lean~4 formalization of Theorems~\ref{thm:main}
and~\ref{prop:Zi} referenced in the footnote on the first page, built
using the Lean~4 proof assistant and the Mathlib library.  All output
produced with these tools was reviewed and verified by the authors,
who take full responsibility for the correctness and content of this
paper.

\end{document}